\documentclass[12pt,english]{smfart}

\usepackage{dsfont,newtxtext,fourier}

\usepackage[text={6.5in,8.5in},centering]{geometry}

\usepackage[dvipsnames]{xcolor}   
\usepackage{xparse}
\usepackage{xr-hyper}
\usepackage[linktocpage=true,colorlinks=true,hyperindex,citecolor=teal,linkcolor=blue]{hyperref}

\DeclareMathOperator{\spec}{\mathrm{Spec}_{\textit{z}}(S)}
\DeclareMathOperator{\id}{\mathrm{Id}(\textit{S})}
\DeclareMathOperator{\mx}{\mathrm{Max}(\textit{S})}
\DeclareMathOperator{\z}{\mathrm{Id}_{\textit{z}}(\textit{S})}

\newcommand{\p}{\mathfrak{p}}

\DeclareMathOperator{\cz}{\mathrm{c}\;\!\!\ell_{\!\textit{z}}}

\newtheorem{theorem}{Theorem}
\newtheorem{proposition}[theorem]{Proposition}
\newtheorem{lemma}[theorem]{Lemma}
\theoremstyle{remark}
\newtheorem{remark}[theorem]{Remark}
\theoremstyle{definition}
\newtheorem{definition}{Definition}
\numberwithin{equation}{section}

\numberwithin{equation}{section}
\renewcommand{\labelenumi}{\upshape(\arabic{enumi})}

\begin{document}
	
\title{Further remarks on $z$-ideals of semirings}

\author{Ahmed Maatallah}

\address{University of Monastir, Faculty of Sciences of Monastir,
Research Laboratory on Algebra, Number Theory and Nonlinear Analysis,
LR 18 ES 15, Boulevard of Environment, Monastir, 5019, Tunisia.}

\email{ahmedmaatallahsd93@gmail.com}
\author{Amartya Goswami}
	
\address{[1]  Department of Mathematics and Applied Mathematics, University of Johannesburg, P.O. Box 524, Auckland Park 2006, South Africa. [2]  National Institute for Theoretical and Computational Sciences (NITheCS), South Africa.}
	
\email{agoswami@uj.ac.za}
	
\subjclass{16Y60, 16D25.}

\keywords{Semiring, $z$-ideal, $z$-prime ideal.}
	
\begin{abstract}
In this note, we revisit certain results from a previous work of the second author concerning \( z \)-ideals of semirings. We demonstrate that the assumption of the semiring being a bzi-semiring can be dispensed with. Additionally, we propose a revised definition of the \( z \)-radical that enables a more general formulation of results from the earlier work.
\end{abstract}
\maketitle

Let \( S \) be a semiring. We denote by \( \id \) the set of all ideals of \( S \), and by \( \mx \) the set of all maximal ideals. An ideal \( \mathfrak{z} \subseteq S \) is called a \emph{\( z \)-ideal} if for every \( x \in \mathfrak{z} \), we have
\[
\bigcap \left\{ \mathfrak{m} \in \mx \mid x \in \mathfrak{m} \right\} \subseteq \mathfrak{z}.
\]
We write \( \z \) for the set of all \( z \)-ideals of \( S \). A proper \( z \)-ideal is said to be \emph{\( z \)-maximal} if it is not properly contained in any other proper \( z \)-ideal. A proper \( z \)-ideal \( \mathfrak{p} \) is said to be \emph{\( z \)-semiprime} if \( \mathfrak{a}^2 \subseteq \mathfrak{p} \) implies \( \mathfrak{a} \subseteq \mathfrak{p} \) for all \( \mathfrak{a} \in \z \). Similarly, \( \mathfrak{p} \) is said to be \emph{\( z \)-prime} if \( \mathfrak{a} \mathfrak{b} \subseteq \mathfrak{p} \) implies \( \mathfrak{a} \subseteq \mathfrak{p} \) or \( \mathfrak{b} \subseteq \mathfrak{p} \) for all \( \mathfrak{a}, \mathfrak{b} \in \z \).

Given \( x \in S \), we write \( \mathcal{M}_x := \{ \mathfrak{m} \in \mx\mid x \in \mathfrak{m} \} \), and define the corresponding \emph{basic \( z \)-ideal} \( \mathfrak{m}_x := \bigcap \mathcal{M}_x \). If every basic \( z \)-ideal in \( S \) is idempotent, we say that \( S \) is a \emph{bzi-semiring}. The \emph{\( z \)-closure} of an ideal \( \mathfrak{a} \in \id \) is defined by
\[
\cz(\mathfrak{a}) := \bigcap \left\{ \mathfrak{z} \in \z \mid \mathfrak{a} \subseteq \mathfrak{z} \right\}.
\]

It was shown in \cite[Theorem 3.7]{1} that the product of two \( z \)-ideals is again a \( z \)-ideal if and only if \( S \) is a bzi-semiring. This result was employed in the proof of \cite[Proposition 4.5]{1}, thereby imposing a restrictive assumption on the ambient semiring. We now show that this condition is superfluous.

\begin{theorem}
\label{wbzi}
An ideal \( \mathfrak{p} \subseteq S \) is \( z \)-prime if and only if it is a prime \( z \)-ideal.
\end{theorem}

\begin{proof}
Every prime \( z \)-ideal is trivially \( z \)-prime. Conversely, let \( \mathfrak{p} \) be a \( z \)-prime ideal and suppose \( \mathfrak{a} \mathfrak{b} \subseteq \mathfrak{p} \) for some \( \mathfrak{a}, \mathfrak{b} \in \id \). By \cite[Lemma 3.14]{1}, we have
\[
\cz(\mathfrak{a}) \cz(\mathfrak{b}) \subseteq \cz(\cz(\mathfrak{a}) \cz(\mathfrak{b})) = \cz(\mathfrak{a} \mathfrak{b}) \subseteq \cz(\mathfrak{p}) = \mathfrak{p}.
\]
Since \( \cz(\mathfrak{a}), \cz(\mathfrak{b}) \in \z \), and \( \mathfrak{p} \) is \( z \)-prime, it follows that \( \cz(\mathfrak{a}) \subseteq \mathfrak{p} \) or \( \cz(\mathfrak{b}) \subseteq \mathfrak{p} \). Hence \( \mathfrak{a} \subseteq \mathfrak{p} \) or \( \mathfrak{b} \subseteq \mathfrak{p} \).
\end{proof}

\begin{remark}
The preceding theorem allows one to remove the bzi-semiring assumption from all results in \cite{1} that relied on \cite[Proposition 4.5]{1}.
\end{remark}

We next show that every 
$z$-ideal is 
$z$-semiprime. For this we first prove the following auxiliary result.

\begin{lemma}
\label{radical}
Every $z$-ideal of a semiring $S$ is radical. 
\end{lemma} 

\begin{proof}
The proof is in fact  part of the claim \cite[Lemma 3.14(13)]{1}. We give, however, another proof. Let $\mathfrak{a}$ be a $z$-ideal of $S$ and let $x \in \sqrt{\mathfrak{a}}$. Then 
$x^n \in \mathfrak{a}$ for some $n \in \mathds{N}$. Clearly, $\mathcal{M}_x=\mathcal{M}_{x^n}$. 
Therefore $x \in \mathfrak{a}$, since $\mathfrak{a}$ is a $z$-ideal. Hence $\mathfrak{a}=\sqrt{\mathfrak{a}}.$
\end{proof}

\begin{proposition}
Let \( \mathfrak{p} \in \z \) and \( \mathfrak{a} \in \id \). If \( \mathfrak{a}^2 \subseteq \mathfrak{p} \), then \( \mathfrak{a} \subseteq \mathfrak{p} \).
\end{proposition}

\begin{proof}
Since \( \mathfrak{p} \) is a \( z \)-ideal and \( \mathfrak{a}^2 \subseteq \mathfrak{p} \), we have
$\mathfrak{a} \subseteq \sqrt{\mathfrak{a}}
=\sqrt{\mathfrak{a}^2} \subseteq 
\sqrt{\mathfrak{p}}=\mathfrak{p}.$
\end{proof}

We now turn to the notion of 
$z$-radicals (\textit{cf}. \cite[Definition 4.10]{1}), which requires refinement. The need for revision arises in \cite[Lemma 4.11(3)]{1} from the observation that the product of two 
$z$-ideals need not be a 
$z$-ideal, and from the observation in Theorem \ref{z-radical} that every 
$z$-ideal is already 
$z$-radical.

\begin{definition}
\label{radd}
The $z$-radical of an ideal $\mathfrak{a}$ of a semiring $S$ is defined by
\[ \sqrt[z]{\mathfrak{a}}= \bigcap_{\mathfrak{p} 
\in \spec} \left\{\mathfrak{p} ~|~ \mathfrak{a} \subseteq \mathfrak{p}\right\}.\]
Whenever $\sqrt[z]{\mathfrak{a}}=\mathfrak{a}$, 
we say $\mathfrak{a}$ is a $z$-radical ideal.
\end{definition} 

The following result constitutes a revised formulation of \cite[Lemma 4.11]{1}. We additionally observe that the inclusions stated in part (3) of that lemma are, in fact, equalities, as clarified in part (4) below.

\begin{lemma}\label{Lemma of z-rad}
Let $\mathfrak{a}$ and $\mathfrak{b}$ be two ideals of a semiring $S$. Then the following holds.

\begin{enumerate}
\item $\sqrt[z]{\mathfrak{a}}$ is a $z$-ideal containing $\mathfrak{a}$.
\item If $\mathfrak{a} \subseteq \mathfrak{b}$, then
$\sqrt[z]{\mathfrak{a}} \subseteq \sqrt[z]{\mathfrak{b}}$.
\item $\sqrt[z]{\sqrt[z]{\mathfrak{a}}}=\sqrt[z]{\mathfrak{a}}$.
\item $\sqrt[z]{\mathfrak{a}\mathfrak{b}}=\sqrt[z]{\mathfrak{a} \cap \mathfrak{b}}=\sqrt[z]{\mathfrak{a}} \cap \sqrt[z]{\mathfrak{b}}$.
\end{enumerate}
\end{lemma}

\begin{proof}
$(1)$ $\sqrt[z]{\mathfrak{a}}$ is a $z$-ideal as an
intersection of $z$-ideals. It is clear that $\mathfrak{a}
\subseteq \sqrt[z]{\mathfrak{a}}$.

$(2)$ Since $\mathfrak{a} \subseteq \mathfrak{b}$, for every 
$\mathfrak{p} \in \spec$, $\mathfrak{b} \subseteq \mathfrak{p}$
implies that $\mathfrak{a} \subseteq  \mathfrak{p}$. Therefore, $\sqrt[z]{\mathfrak{a}} \subseteq \sqrt[z]{\mathfrak{b}}$.

$(3)$ It is clear that for every 
$\mathfrak{p} \in \spec$, $\mathfrak{a} \subseteq \mathfrak{p}$ if and only if 
$\sqrt[z]{\mathfrak{a}} \subseteq \mathfrak{p}$. Then, by Definition \ref{radd}, 
$\sqrt[z]{\sqrt[z]{\mathfrak{a}}}=\sqrt[z]{\mathfrak{a}}$.

$(4)$ Using $(2)$, we have
$\sqrt[z]{\mathfrak{a}\mathfrak{b}} \subseteq \sqrt[z]{\mathfrak{a} 
\cap \mathfrak{b}} \subseteq  \sqrt[z]
{\mathfrak{a}} \cap \sqrt[z]{\mathfrak{b}}$.
Now, it is enough to show that $\sqrt[z]
{\mathfrak{a}} \cap \sqrt[z]{\mathfrak{b}} \subseteq 
\sqrt[z]{\mathfrak{a}\mathfrak{b}}$. Let
$\mathfrak{p} \in \spec$ such that $\mathfrak{a} \mathfrak{b} \subseteq \mathfrak{p}$. Since $\mathfrak{p}$
is prime,  either $\mathfrak{a}\subseteq \p$ or $\mathfrak{b} \subseteq \mathfrak{p}.$ Thus, $\sqrt[z]{\mathfrak{a}}$ or $\sqrt[z]{\mathfrak{b}} \subseteq \mathfrak{p}.$ Then $\sqrt[z]{\mathfrak{a}} \cap \sqrt[z]{\mathfrak{b}} \subseteq \mathfrak{p}.$
\end{proof}

Thanks to Lemma~\ref{radical}, the following result provides a revised version of \cite[Theorem 4.16]{1}.

\begin{theorem}\label{z-radical}
For any ideal $\mathfrak{a}$ of a semiring $S$, the following are equivalent.
\begin{enumerate}
\item $\mathfrak{a}$ is z-ideal..
\item $\mathfrak{a}$ is $z$-radical.
\end{enumerate}
\end{theorem}
\begin{proof}
It is sufficient to prove $(1) \Rightarrow (2)$. We know that $\mathfrak{a} \subseteq \sqrt[z]{\mathfrak{a}}$. Let $x \notin \mathfrak{a}$. Then $X=\{x^n ~|~ n \in \mathds{N}\}$ is a multiplicatively closed subset of $S$ disjoint with 
$\mathfrak{a}$. Indeed, if $x^n \in \mathfrak{a}$, for some $n \in \mathds{N}$, then $x \in \sqrt{\mathfrak{a}}=\mathfrak{a}$, which
by  Lemma \ref{radical} is absurd. Hence, $x \notin \sqrt[z]{\mathfrak{a}}$, by \cite[Lemma 4.14.]{1},  
\end{proof}

In \cite[Proposition 4.23]{1}, it is proved  that 
a $z$-ideal of a bzi-semiring is 
$z$-prime if and only if it is $z$-semiprime and $z$-strongly irreducible. Thanks to Theorem \ref{wbzi} and Lemma \ref{radical}, the statement can be simplified.
 
\begin{proposition}
An ideal of a semiring is $z$-prime if and only if it is a $z$-strongly irreducible.
\end{proposition}

\begin{proof}
Clearly, every $z$-prime is a $z$-strongly
irreducible ideal. Conversely, let 
$\mathfrak{p}$ be a $z$-strongly irreducible
ideal. If $\mathfrak{a} \mathfrak{b} \subseteq \mathfrak{p}$, for some $\mathfrak{a}$, $\mathfrak{b} \in \z$. 
Then by Theorem \ref{z-radical} and Lemma \ref{Lemma of z-rad}
$ \sqrt[z]{\mathfrak{a} \cap \mathfrak{b}}=
 \sqrt[z]{\mathfrak{a}  \mathfrak{b}} 
 \subseteq \sqrt[z]{\mathfrak{p}}=\mathfrak{p}$.
 Therefore, $\mathfrak{a} \cap \mathfrak{b}
 \subseteq \mathfrak{p}.$ By \cite[Proposition 4.19.]{1}, $\mathfrak{p}$ is strongly irreducible, and hence,  either $\mathfrak{a} \subseteq \mathfrak{p}$
 or $\mathfrak{b} \subseteq \mathfrak{p}.$
\end{proof}

\begin{remark} 
We finish this note with a few further observations.
\begin{enumerate}
\item[$\bullet$] The proof of \cite[Lemma 4.2]{1} is straightforward, as every maximal ideal is $z$-maximal.

\item[$\bullet$] The proof of the converse of \cite[Proposition 
4.3]{1} can be simplified as follows.
Let $\mathfrak{m}$ be a $z$-maximal ideal. Suppose that $\mathfrak{m}$ is not maximal, then $\mathfrak{m}$ is properly contained
in a maximal ideal $\mathfrak{r}$. However,
every maximal ideal is $z$-maximal, 
which leads to a contradiction.

\item[$\bullet$] Since every $z$-ideal is radical and that  $\cz(\mathfrak{a})$ is a $z$-ideal, it follows that \cite[Lemma 3.14(12)]{1} will become $\cz(\mathfrak{a})=\cz(\sqrt{\mathfrak{a}})$.
\end{enumerate}
\end{remark}

\end{document}